\documentclass[12pt]{amsart}
\usepackage[paper = a4paper,left=4cm,top=1.5cm,right=1.5cm,bottom = 3cm,includehead]{geometry}
\usepackage{amssymb, newcent}
\usepackage{amsmath}
\usepackage{mathrsfs}
\usepackage{graphics}
\usepackage{graphicx}
\usepackage{amsthm}
\usepackage[all]{xy}
\newtheorem{theorem}{Theorem}[section]

\newtheorem{corollary}[theorem]{Corollary}

\theoremstyle{remark}

\newtheorem{remark}[theorem]{Remark}
\newtheorem{definition}[theorem]{Definition}

\def\P{\mathbb{P}}
\def\RR{\mathbb{R}}
\def\ZZ{\mathbb{Z}}
\def\V{\mathfrak{V}}
\def\a{\mathfrak{a}}
\def\b{\mathfrak{b}}
\def\c{\mathfrak{c}}
\def\aa{\tilde{a}}
\def\bb{\tilde{b}}

\def\i{\tilde{\imath}}
\def\j{\tilde{\jmath}}
\def\k{\tilde{k}}
\def\s{\tilde{s}}
\def\SS{\tilde{S}}
\def\Diff{\mathrm{Diff}}
\def\PGL{\mathrm{PGL}}
\def\PPi{\tilde{\Pi}}
\def\MM{\tilde{M}}
\def\GGamma{\tilde{\Gamma}}
\def\q{\tilde{q}}
\def\N{n_{0}}

\def\div{\mathrm{div}}
\def\f{\tilde{f}}

\def\ord{\mathrm{ord}}

\begin{document}
\title{Projective Connections and Schwarzian derivatives for Supermanifolds, and Batalin-Vilkovisky operators}
\author{Jacob George}
\address{School of Mathematics - University of Manchester, Oxford Road, Manchester M13 9PL}
\email{jgeorge@maths.manchester.ac.uk}

\begin{abstract}

We extend the notion of a Thomas projective connection (a projective
equivalence class of linear connections) for supermanifolds.  As a
by-product, we arrive at a generalisation of the multidimensional
Schwarzian derivative for the super case which was previously
unknown.

This is combined with our previous construction of a Laplacian on
the algebra of densities for a projectively connected manifold and
allows us to study Batalin-Vilkovisky type operators and the
relation between projective connections and odd Poisson structures.

\end{abstract}

\date{\today}
\maketitle
\section{Introduction}This paper concerns the links between projective classes on supermanifolds, second order differential operators acting on functions and densities and symmetric biderivations (called \emph{brackets}).

In the purely even case, projective classes are defined as equivalence classes of linear connections whose geodesics are the same up to reparametrisation.  In \cite{bialopap}, it was discovered that on manifolds with a projective class, any symmetric tensor $S^{ab}$ could be extended canonically to give a second order differential operator acting on functions (the \emph{projective Laplacian}).  Rephrasing this in terms of brackets, on a manifold with a projective class, any bracket has a canonical generating operator.  Further, there is a fundamental relation between the algebra of densities of all weights and projective classes.  One particular manifestation of this given in \cite{me2} is that brackets on functions can canonically be extended to brackets on densities for manifolds with a projective class.

Of course, one would hope to apply to constructions to familiar geometric situations.  The obvious example is that of Poisson manifolds.  However, requiring the brackets considered to be symmetric immediately precludes this possibility.  Switching to supermanifolds however, allows for odd brackets.  To any odd symmetric bracket $\{\,\cdot\, ,\,\cdot\,\}$ is associated an odd biderivation $[\,\cdot\, ,\,\cdot\,]$ which is antisymmetric with respect to reversed parity.  A Jacobi identity can be required of $[\,\cdot\, ,\,\cdot\,]$, the fulfillment of which results in an \emph{odd Poisson} (or \emph{Schouten}) algebra.

The purpose of this paper is to generalise the framework developed in \cite{me2, bialopap} to supermanifolds and to apply the results obtained to the study of odd Poisson brackets.  This involves defining projective classes and deriving a super version of the classical Thomas construction.  As a by-product, we obtain a multidimensional Schwarzian derivative for supermanifolds which generalises that given by Ovsienko and Tabachnikov (\cite{OvsienkoTabachnikov}).  For a supermanifold with a projective class, we construct super analogues of the projective Laplacian acting on functions and a canonical extensionsion of any bracket on functions to a bracket on densities.

In the penultimate section, we apply the preceding constructions to odd Poisson brackets.  This immediately gives conditions for the projective Laplacian associated with a Poisson structure to be a Batalin-Vilkovisky operator.  For densities, we reiterate the geometric interpretation of the components of a bracket on densities in terms of upper connections. Having done this we give conditions for the projective extension of an odd Poisson bracket to densities to also be an odd Poisson bracket.  Following \cite{oloii}, it is possible to interpret these conditions in terms of the Poisson-Lichernowicz operator.  Finally, we discuss some possible applications.

\section{The Classical case: a quick survey}
We give a whistlestop tour of selected results from \cite{me2,bialopap} omitting details.  First, a few definitions.
\begin{definition}
Let an $n$-dimensional manifold $M$ be endowed with two linear connections $\nabla$, $\bar{\nabla}$ with coefficients $\Gamma^{k}_{ij}$ and $\bar{\Gamma}^{k}_{ij}$ respectively.  The two connections belong to the same \emph{projective class} if either $\nabla$ and $\bar{\nabla}$ have the same geodesics up to reparametrisation or equivalently
\begin{equation}\label{projcond}\Pi^{k}_{ij} := \Gamma^{k}_{ij} - \frac{1}{n+1}(\delta^{k}_{i}\Gamma^{s}_{sj} + \delta^{k}_{j}\Gamma^{s}_{is}) = \bar{\Gamma}^{k}_{ij} - \frac{1}{n+1}(\delta^{k}_{i}\bar{\Gamma}^{s}_{sj} + \delta^{k}_{j}\bar{\Gamma}^{s}_{is}) =:\bar{\Pi}^{k}_{ij}.
\end{equation}
\end{definition}
The quantities $\Pi^{k}_{ij}$ are constitute an invariant of the projective class and will often be referred to as the projective class themselves.  Various authors refer to projective classes as \emph{projective connections} and \emph{projective structures}. Here we adopt the chosen wording to avoid confusion with other distinct notions of projective connection and projective structure.  For a more comprehensive guide to projective connections in their various guises, see for instance \cite{crampin} and the references contained therein.
Consider a manifold $M$ equipped with a tensor field $S^{ij}$.  This defines a natural `index raising' map $S^{\sharp}:T^{*}M \to TM$ by $S^{\sharp}(\omega_{i}dx^{i}) = S^{ij}\omega_{j}\partial_{i}$.
\begin{definition} Let $M$ be a manifold with a tensor field $S^{ij}$ and $E\to M$ a vector bundle over $M$.  An \emph{upper connection} or \emph{contravariant derivative} is an $\RR$-bilinear map $\nabla:\Gamma(T^{*}M)\times \Gamma(E)\to\Gamma(E)$ satisfying
\[\nabla^{f\omega}\sigma = f\nabla^{\omega}\sigma\mbox{ and }\nabla^{\omega}f\sigma = S^{\sharp}(\omega)(f)\sigma + f\nabla^{\omega}\sigma\]
for $f\in C^{\infty}(M)$ and any $1$-form $\omega$.
\end{definition}
On manifolds with projective classes, we have the following result from \cite{bialopap}.
\begin{theorem}\label{first result}  Let $S^{ij}$ be a symmetric tensor field and $\Pi^{k}_{ij}$ a projective class on $M$.
\begin{enumerate}
\item The expression \[S^{ij}\partial_{i}\partial_{j} + \left(\frac{2}{n+3}\partial_{j}S^{ij} - \frac{n+1}{n+3}S^{jk}\Pi^{i}_{jk}\right)\partial_{i}\]
defines an invariant second order differential operator on $C^{\infty}(M)$ (the \emph{projective Laplacian}).
\item There is an associated canonical upper connection over $S^{ij}$ in the bundle of volume forms whose coefficients are
\[\gamma^{i} = \frac{n+1}{n+3}(\partial_{j}S^{ij} + S^{jk}\Pi^{i}_{jk}).\]
\end{enumerate}
\end{theorem}
We now turn to the algebra of densities and its relation with projective classes.  First we run through the requisite definitions.
\begin{definition}A \emph{density of weight $\lambda\in\RR$} on a manifold $M$ is an object formally expressed in local coordinates as $\phi(x)|Dx|^{\lambda}$ where $Dx$ is the local coordinate volume form.  The space of $\lambda$-densities is denoted $\V_{\lambda}(M)$.  The \emph{algebra of densities} is the vector space $\V(M) := \bigoplus_{\lambda\in\RR}\V_{\lambda}(M)$ with the product defined as $\phi(x)|Dx|^{\lambda}\cdot\psi(x)|Dx|^{\mu} = \phi(x)\psi(x)|Dx|^{\lambda + \mu}$ on homogeneous elements and extended by linearity.  The \emph{weight operator} $w:\V(M)\to\V(M)$ is the linear operator specified as having $\V_{\lambda}(M)$ as its $\lambda$-eigenspace.
\end{definition}

Among the primary objects of study in \cite{oloii, bialopap, me2} were so called brackets.
\begin{definition}\label{bracketdefn}A \emph{bracket} on a commutative associative unital algebra $A$ is a symmetric biderivation on $A$.  For a manifold $M$, brackets on $C^{\infty}(M)$ will be referred to as \emph{brackets on $M$}.  A bracket $\{\,\cdot\, , \,\cdot\,\}$ on $A$ is \emph{generated} by an operator $\Delta:A\to A$ if $\{a,b\} = \Delta(ab) - a\Delta(b) - \Delta(a)b + ab\Delta(1)$.
\end{definition}
In \cite{oloii}, the foundational result was to give for algebras with an invariant scalar product, a canonical generating operator for each bracket.
\begin{theorem}\label{Koszul}(\cite{oloii})  Let $A$ be an algebra as above.  Let $\{\,\cdot\, , \,\cdot\,\}$ be a bracket on $A$.
\begin{enumerate}
\item Any operator generating $\{\,\cdot\, , \,\cdot\,\}$ is of order two in the algebraic sense.  Any two operators generating $\{\,\cdot\, , \,\cdot\,\}$ differ by an operator of order one.
\item Let  $A$ be endowed additionally with an invariant scalar product. $\langle\,\cdot\, ,\,\cdot\,\rangle$\footnote{i.e. $\langle ab,c\rangle = \langle a,bc\rangle\:\forall a,b,c\in A$}  Then generating any bracket, there is an unique operator $\Delta$ which is self adjoint and satisfying $\Delta(1) = 0$.
\end{enumerate}
\end{theorem}
Brackets on $M$ are synonymous with symmetric tensor fields with two upper indices by defining $\{f,g\} = S^{ij}\partial_{i}f\partial_{j}g$.  In the light of this, the first part of Theorem \ref{first result} can be rephrased as associating a canonical generating operator to any bracket on a manifold with a projective class.  We can also consider brackets on $\V(M)$.  We will always assume these to be homogenous with respect to the $\RR$-grading.  For a bracket of weight $\lambda$, these are specified by a triple $(S^{ij},\gamma^{i},\theta)$ as follows
\begin{equation}\label{bracketcomponents}
S^{ij}|Dx|^{\lambda} = \{x^{i},x^{j}\}, \gamma|Dx|^{\lambda+1} = \{x^{i},|Dx|\}, \theta|Dx|^{\lambda+2} = \{|Dx|,|Dx|\}.
\end{equation}
Moreover, there is an invariant scalar product on $\V(M)$, given by integrating in the case that the weights of the two arguments sum to one and set to zero otherwise.  The application of Theorem \ref{Koszul} to this situation was the main thrust of \cite{oloii}.

In \cite{me2}, the following results were obtained.
\begin{theorem}\label{main}
Let $M$ be a manifold of dimension $n>1$ endowed with a projective class $\Pi^{k}_{ij}$ and a tensor density $S^{ij}$ of weight $\lambda\neq\frac{n+2}{n+1},\frac{n+3}{n+1}$ (i.e. $S^{ij}|Dx|^{\lambda}$ transforms as a tensor).  Then there is a canonical operator extending $S^{ij}$.
\end{theorem}
\begin{theorem}\label{bracketextension}
Let $M$ be a manifold with a projective class.  Then any bracket on $M$ can be extended to a bracket on $\V(M)$.
\end{theorem}
The notion of `extending' a bracket makes sense here since a bracket on $M$ is a bracket on the $0$-graded slice of $\V(M)$, $\V_{0}(M) = C^{\infty}(M)$.

\section{The generalisation for supermanifolds}
The machinery developed and used in \cite{bialopap} and \cite{me2} can be generalised wholesale for supermanifolds.

\subsection{Projective Classes}
In defining projective classes, it is best to adopt the geometric interpretation of $\Pi^{k}_{ij}$ as defining a projective Ehresmann connection on $TM$.  A linear connection on $TM$ induces a fibre bundle connection on $\P (TM)$ as the image of the horizontal distribution on $TM$ under the derivative of the projection map.  Two linear connections on $TM$ in the same proejctive class induce the same connection on $\P (TM)$.  For a more explicit treatment see \cite{me2}.

Replicating this process in the super case involves defining the traceless part of linear connection coefficients. To this end, consider an $n|m$-dimensional supervector space $V$.  Let $e_{i}$ be a basis of $V$ and $e^{i}$ denote the dual basis of $V^{*}$.  The trace to be defined should be a map $\Sigma^{2}V^{*}\otimes V\to V^{*}$, where $\Sigma^{2}V^{*}$ denotes the symmetric square of $V^{*}$. We define $\div:\Sigma^{2}V^{*}\otimes V \to V^{*}$ on basis vectors by\footnote{Throughout we adopt the convention that vectors are written $v^{i}e_{i}$ and covectors as $e^{i}v_{i}$.  The parity of any object $X$ will be denoted $\tilde{X}$.}
\[\div(e^{i}\vee e^{j}\otimes e_{k}) = e^{j}\delta^{i}_{k}(-1)^{\i(\j + \k)} + e^{i}\delta^{j}_{k}(-1)^{\j\k}\]
where $\i$ denotes the parity of $e_{i}$, $\delta^{k}_{i}$ refers to the Kronecker delta and $\vee$ is the symmetric product.  On arbitrary elements of $\Sigma^{2}V^{*}\otimes V$ therefore, $\div(e^{i}\vee e^{j}A^{k}_{ji}\otimes e_{k}) = e^{i}2A^{j}_{ij}(-1)^{\j}$.  The map $\div$ will be the notion of `trace' used here.  It is nothing but the super trace of $A^{k}_{ij}$ viewed as an endomorphism valued form on $V$.

There is also a natural injection $j:V^{*}\to \Sigma^{2}V^{*}\otimes V$ given by
\[\phi\mapsto \phi\vee e^{i}\otimes e_{i}.\]
Since the composition $\div\circ j$ is simply $(n-m+1)id$, the projection of $A\in\Sigma^{2}V^{*}\otimes V$ onto its traceless part is $A - \frac{j\circ \div (A)}{n-m+1}$.  Explicitly, in terms of the bases of $V$ and $V^{*}$ given above, the trace free part of $e^{i}\vee e^{j}A^{k}_{ji}\otimes e_{k}$ is
\[e^{i}\vee e^{j}\left(A^{k}_{ji} - \frac{1}{n-m+1}(A^{s}_{js}\delta^{k}_{i}(-1)^{\s} + A^{s}_{is}\delta^{k}_{j}(-1)^{\i\j+\s})\right)\otimes e_{k}.\]  This puts us in a position to define projective classes.
\begin{definition}Two symmetric linear connections $\nabla$ and $\bar{\nabla}$ on an $n|m$-dimensional supermanifold define the same \emph{projective class} if their coefficients $\Gamma^{k}_{ij}$ and $\bar{\Gamma}^{k}_{ij}$ satisfy
\begin{equation}\label{supercoeffs}
\begin{split}
\Pi^{k}_{ji} :=& \Gamma^{k}_{ji} - \frac{1}{n-m+1}\left(\Gamma^{s}_{js}\delta^{k}_{i}(-1)^{\s} + \Gamma^{s}_{is}\delta^{k}_{j}(-1)^{\i\j+\s}\right)\\ &= \GGamma^{k}_{ji} - \frac{1}{n-m+1}\left(\GGamma^{s}_{js}\delta^{k}_{i}(-1)^{\s} + \GGamma^{s}_{is}\delta^{k}_{j}(-1)^{\i\j+\s}\right) =: \PPi^{k}_{ji}.
\end{split}
\end{equation}
\end{definition}
Retrospectively, the formula for $\Pi^{k}_{ij}$ on a supermanifold is easily recognisable as a super version of (\ref{projcond}), replacing the dimension $n$ by $n-m$ and trace with super trace.
\subsubsection*{Multidimensional super-Schwarzian Derivatives}
The notation $\div$ and $j$ used above is a deliberate reference to Ovsienko and Tabachnikov's book \cite{OvsienkoTabachnikov}, the construction of a `trace free' part mirorring in the super setting the treatment given there.  The purpose of defining a trace free part there was in order to construct a multidimensional version of the Schwarzian derivative.  This can be defined as a cocycle measuring  the failure of the quantities $\Pi^{k}_{ij}$ to transform as a tensor.  The formula (\ref{supercoeffs}) then automatically gives a definition of the multidimensional Schwarzian derivative in the super case.  Noting that
\[\partial_{i}\log J = \frac{\partial^{2}\bar{x}^{\sigma}}{\partial x^{i}\partial x^{k}}\frac{\partial x^{k}}{\partial\bar{x}^{\sigma}}(-1)^{\k},\mbox{ where }\bar{x}^{\alpha} = \bar{x}^{\alpha}(x^{i})\mbox{ and }J = \left|\det\frac{D\bar{x}}{Dx}\right|,\]
the coordinate expression for the Schwarzian derivative of a map $f:\RR^{n|m}\to\RR^{n|m}$ is
\begin{equation}\mathfrak{S}^{k}_{ij}(f) = \frac{\partial^{2}f^{\sigma}}{\partial x^{i}\partial x^{j}}\frac{\partial x^{k}}{\partial f^{\sigma}} - \frac{1}{n-m+1}\left(\delta^{k}_{i}\partial_{j}\log J_{f} + \delta^{k}_{j}\partial_{i}\log J_{f} (-1)^{\i\j}\right).
\end{equation}
Ovsienko and Tabachnikov give a relation between their multidimensional Schwarzian derivative and the classical $1$-dimensional version in terms of $\PGL$-relative cocycles in the group cohomologies of $\Diff(\RR P^{1})$ and $\Diff(\RR P^{n})$.  It remains uninvestigated whether Duval and Michel's definition (see \cite{DuvalMichel}) in the case of the $1|m$-dimensional super-circle and the definition given here stand in a similar relation.

\subsection{The Projective Laplacian and upper volume connections}
For supermanifolds endowed with a projective class, there is an analogue of Theorem \ref{first result}.
\begin{theorem}\label{supersquare}
Let $M$ be an $n|m$-dimensional supermanifold with a projective class $\Pi^{k}_{ji}$ and a tensor field $S^{ij}$.
\begin{enumerate}
\item Then \begin{equation}\label{superprojectivelaplacian}
S^{ij}\partial_{j}\partial_{i} + \left(\frac{2}{n-m+3}\partial_{j}S^{ji}(-1)^{\j(\SS+1)} - \frac{n-m+1}{n-m+3}S^{jk}\Pi^{i}_{kj}\right)\partial_{i}
\end{equation}
is an invariant differential operator on $M$.
\item There is an associated connection in the bundle of volume forms whose coefficients are
\[\gamma^{i} = \frac{n-m+1}{n-m+3}\left(\partial_{j} S^{ji}(-1)^{\j(\SS+1)} + S^{jk}\Pi^{i}_{kj}\right).\]
\end{enumerate}
\end{theorem}
\subsection{Densities, Brackets and Operators}
The definition of the algebra of densities carries over word for word, now taking $|Dx|$ to be the local Berezin volume element and $J$ to be the modulus of the Berezinian of the change of coordinates.  Brackets are defined again as being symmetric biderivations.  In this context however, the bracket itself has a parity which `sits at the opening bracket'.  A bracket of parity $\epsilon$ on a $\ZZ_{2}$-graded algebra $A$ is then generated by an operator $\Delta: A\to A$ if for any $a,b\in A$,
\[\{a,b\} = \Delta(ab) - a\Delta(b)(-1)^{\a\tilde{\Delta}} - \Delta(a)b + ab\Delta(1)(-1)^{(\aa+\bb)\tilde{\Delta}}.\]
Theorem \ref{Koszul} holds for $\ZZ_{2}$-graded algebras, in fact this was the original context in \cite{oloii}.  Taking (\ref{bracketcomponents}) as the components of a bracket on $\V(M)$, the unique self-adjoint constant-free generating operator guaranteed by the theorem and given in \cite{oloii} is
\begin{equation}\label{THsupercanonical}
\begin{split}
\Delta = |Dx|^{\lambda}&\left(S^{ij}\partial_{j}\partial_{i} + 2\gamma^{i}w\partial_{i} + \theta w^{2} + \left(\partial_{j}S^{ji}(-1)^{\j(\epsilon + 1)} + (\lambda - 1)\gamma^{i}\right)\partial_{i}\vphantom{\left(\partial_{k}\gamma^{k}(-1)^{\k(\epsilon + 1)} + (\lambda - 1)\theta\right)}\right.\\&\left. + \left(\partial_{k}\gamma^{k}(-1)^{\k(\epsilon + 1)} + (\lambda - 1)\theta\right)w\right).
\end{split}
\end{equation}
Note that the parity of $S$, $\gamma$ and $\theta$ is $\epsilon$.

\subsection{The Thomas Construction}
One of the primary architects of the theory of projective classes, T. Y. Thomas sought a way of studying the projective geometry of a manifold in terms of the affine geometry of another related manifold.  This is the content of what I will call the \emph{Thomas construction}.  Precisely, associated to any $n$-dimensional manifold $M$, there is a manifold $\MM$, of dinmension $n+1$ such that any projective class on $M$ can be lifted canonically to a linear conenction on $\MM$.  The manifold $\MM$ also has the property that densities on $M$ can be viewed as functions on $\MM$ in a natural way.  This was one of the main observations of \cite{bialopap}.  For specifics concerning the Thomas construction, see \cite{bialopap,me2} or Thomas' original paper \cite{Thomas2}.

In analogy with the classical Thomas construction, for supermanifolds $\MM$ is constructed from $M$ adding an extra even coordinate to each chart. Let $(x^{1},\dots,x^{n+m})$ be a local coordinate system on $M$.  Then $x^{0},\dots, x^{n+m}$ is a local coordinate system on $\MM$.  Under a change of coordinates $\bar{x}^{i} = \bar{x}^{i}(x^{j})$ on $M$, the transformation law for $x^{0}$ is $\bar{x}^{0}:= x^{0} + \log J$ where $J$ is the absolute value of the Berezinian of the coordinate change.

\begin{theorem}A projective class $\Pi^{k}_{ij}$ on $M$ defines a linear connection $\GGamma^{\c}_{\a\b}$ on $\MM$ via
\begin{equation}\begin{split}
\GGamma^{k}_{ij} = \Pi^{k}_{ij},\quad \GGamma^{\c}_{0\a} = \GGamma^{\c}_{\a 0} = \frac{-\delta^{\c}_{\a}}{n-m+1}\\
\GGamma^{0}_{ji} = \frac{n-m+1}{n-m-1}\left(\partial_{q}\Pi^{q}_{ji} - \Pi^{p}_{qj}\Pi^{q}_{pi}\right)(-1)^{\q(1+\i+\j)}.
\end{split}
\end{equation}
Here Gothic indices range from $0$ to $n+m$ while Roman ones range from $1$ to $n+m$.
\end{theorem}
As in the purely even case, $\MM\to M$ is the oriented frame bundle of the Berezinian bundle.  A density $\phi(x)|Dx|^{\lambda}$ can be viewed as a function $\phi(x)e^{\lambda x^{0}}$ on $\MM$, $\V(M)$ consequently being a subalgebra of $C^{\infty}(\MM)$.  The weight operator corresponds to the vector field $\frac{\partial}{\partial x^{0}}$ (see \cite{bialopap,me2} for the classical case).

Using this new version of Thomas's construction, super versions of the main theorems of \cite{me2} can be proved.  These are summarised in the next theorem.
\begin{theorem}
Let $M$ be an $n|m$-dimensional supermanifold with a projective class $\Pi^{k}_{ji}$.
\begin{enumerate}
\item  For $(n-m)\notin\{1, -1, -2\}$ the formulae below define a projective class $\PPi^{\c}_{\a\b}$ on $\MM$.
\begin{equation}
\begin{split}
\PPi^{k}_{jk} = \Pi^{k}_{ji},&\quad\quad \PPi^{k}_{j0} = \PPi^{k}_{0j} = \frac{\delta^{k}_{j}}{(n-m+1)(n-m+2)},\\
\PPi^{0}_{i0} = \PPi^{0}_{0i} = 0,&\quad\quad \PPi^{0}_{ji} =\frac{n-m+1}{n-m-1}\left(\partial_{q}\Pi^{q}_{ji} - \Pi^{p}_{qj}\Pi^{q}_{pi}\right)(-1)^{\q(1+\i+\j)},\\
\PPi^{k}_{00} = 0,&\quad\quad\PPi^{0}_{00} = \frac{n-m}{(n-m+1)(n-m+2)}.
\end{split}
\end{equation}
\item  For $(n-m)\notin\{1, -1, -4\}$ Let $\V(M)$ be endowed with a bracket defined by a triple $(S^{ij},\gamma^{i},\theta)$ as per (\ref{bracketcomponents}).  Then there is a canonical generating operator for this bracket given by
\begin{equation}
\begin{split}\Delta = e^{\lambda x^{0}}&\left(S^{ij}\partial_{j}\partial_{i} + 2\gamma^{i}\partial_{i}\partial_{0} + \theta\partial_{0}^{2}\vphantom{\frac{1^{4}}{2}}\right.\\ + &\left(\frac{2}{\N+1}\partial_{j}S^{ji}(-1)^{\j(\SS + 1)} + \frac{2(\lambda(\N+1) + 1)}{(\N+1)(\N+4)}\gamma^{i} - \frac{\N+2}{\N+4}S^{jk}\Pi^{i}_{kj}\right)\partial_{i}\\ +&\left.\left(\frac{2}{\N+4}\partial_{k}\gamma^{k}(-1)^{\k(\SS + 1)} + \frac{2\lambda(\N+1) - (\N)}{(\N+1)(\N+4)}\theta - \frac{(\N+2)}{(\N+4)}S^{jk}B_{kj}\right)\partial_{0}\right)\\
\end{split}
\end{equation}
where $B_{kj} = \frac{\N + 1}{\N-1}\left(\partial_{q}\Pi^{q}_{kj} + \Pi^{p}_{qk}\Pi^{q}_{pj}\right)(-1)^{\q(1+\i+\j)}$ and $\N = n-m$.
\end{enumerate}
\end{theorem}
\begin{proof}The first result is obtained by taking the projective class of the linear connection on $\MM$ given by the Thomas construction.  The second follows by then applying the formula for the projective Laplacian to $\MM$ with the bracket on $\V(M)$ considered as a bracket on $\MM$ and the projective class just constructed.
\end{proof}
\begin{remark}In the classical case, the dimension of the manifold had to be greater than one for the formula to have no singularities.  In the super context, $n-m$ can be negative leading to more singular cases.  The precise implications of these values are yet to be investigated.
\end{remark}
We suppose that this operator is self adjoint and therefore equal to the canonical operator (\ref{THsupercanonical}).  This yields corresponding expressions
\begin{equation}
\begin{split}
\gamma^{i} = \frac{n-m+1}{(n-m+3) - \lambda(n-m+1)}\left(\partial_{j}S^{ij}(-1)^{\j(\SS +1)} + S^{jk}\Pi_{kj}^{i}\right)\\
\theta = \frac{n-m+1}{(n-m+2) - \lambda(n-m+1)}\left(\partial_{k}\gamma^{k}(-1)^{\k(\SS + 1)} - S^{kj}B_{jk}\right)
\end{split}
\end{equation}
These observations lead to the super analogue of Theorem \ref{main} and Theorem \ref{bracketextension}.
\begin{theorem}
\label{super main}Let $M$ be an $n|m$-dimensional supermanifold for $n-m \neq 1,-1,-2,-4$ equipped with a projective class and a tensor density $S^{ij}$ of weight $\lambda\neq\frac{n-m+2}{n-m+1},\frac{n-m+3}{n-m+1}$.  There is a canonical second order differential operator of weight $\lambda$ acting on densities extending $S^{ij}$.
\end{theorem}
\begin{corollary}\label{superbracketextension}
Let $M$ be an $n|m$-dimensional with $(n-m)\neq 1, -2, -3$ supermanifold equipped with a projective class.  Then every bracket on $M$ of can be extended canonically to a bracket on $\V(M)$.
\end{corollary}

\section{Applications to Odd Poisson and BV-structures}
All the brackets above have been symmetric.  This immediately precludes the possibility of applying these constructions to give extensions and generating operators for Poisson brackets.  However, in the case of supermanifolds, symmetric brackets are related to odd Poisson structures.  Much of the explanation that follows is taken from \cite{oloi,oloii,Graded}.  From now, all objects considered are implicitly assumed to be `super'.
\begin{definition}
Let $A$ be an associative algebra.  An \emph{odd Poisson structure} (or \emph{Schouten structure}) on $A$ is an odd bilinear map $[\,\cdot\, ,\,\cdot\,]:A\times A \to A$ satisfying
\begin{eqnarray}
&[a,b] = -(-1)^{(\aa+1)(\bb + 1)}[b,a]&\label{shiftedantisymmetry}\\
&[a,[b,c]] = [[a,b],c] + (-1)^{(\aa + 1)(\bb + 1)}[b,[a,c]]\label{Jacobi}&\\
&[a,bc] = [a,b]c + (-1)^{(\aa + 1)\bb}b[a,c]\label{derivation}&
\end{eqnarray}
If there is an operator $\Delta:A\to A$ such that
\begin{equation}[a,b] = (-1)^{\aa}\left(\Delta(ab) - \Delta(f)g - (-1)^{\aa}a\Delta(b)\right)\label{BVgen}\end{equation}
and such that $\Delta^{2} = 0$, then $(A,[\,\cdot\, ,\,\cdot\,], \Delta)$ is called a \emph{Batalin-Vilkovisky algebra}.
\end{definition}

At first glance, the brackets considered in Definition \ref{bracketdefn}, being symmetric, have no hope of defining odd Poisson structures.  Consider however an odd symmetric bracket $\{\,\cdot\,,\,\cdot\,\}$.  Defining $[a,b]:= (-1)^{\aa}\{a,b\}$, we see that $[\,\cdot\, ,\,\cdot\,]$ satisfies property (\ref{shiftedantisymmetry}) above.  Since $\{\,\cdot\,,\,\cdot\,\}$ is a biderivation, $[\,\cdot\, ,\,\cdot\,]$ satisfies (\ref{derivation}) automatically.  The only condition therefore for $[\,\cdot\, ,\,\cdot\,]$ to define an odd Poisson structure is the Jacobi identity with respect to shifted parity (\ref{Jacobi}).

Consider now a bracket on $C^{\infty}(M)$ specified by a master Hamiltonian $S = S^{ab}p_{b}p_{a}\in C^{\infty}(T{*}M)$.  By this it is meant that
\begin{equation}
\{f,g\} = ((S,f),g) = S^{ab}\partial_{b}f\partial_{a}g(-1)^{\aa(\f + 1)}
\end{equation}
where $(\,\cdot\, ,\,\cdot\,)$ is the canonical even Poisson bracket on $T^{*}M$.  Indeed, it is well known that the Jacobi identity (\ref{Jacobi}) is satisfied by $[a,b] := (-1)^{\aa}\{a,b\}$ if and only if
\begin{equation}
(S,S) = 0.
\end{equation}
This condition of course only has moment for odd brackets - for even brackets, it becomes vacuous.  Given this, we investigate what can be said about manifolds endowed with projective classes.

\subsection{Projective Laplacian as a Batalin-Vilkovisky operator}
Let $M$ be an odd Poisson manifold with bracket $[\,\cdot\,,\,\cdot\,]$ and define the odd symmetric bracket $\{a,b\} = (-1)^{\aa}[a,b]$ as above.  The condition that a constant free operator $\Delta$ on $C^{\infty}(M)$ satisfies (\ref{BVgen}) is the same as requiring that is generate a the bracket $\{\,\cdot\,,\,\cdot\,\}$ in the usual sense.  Suppose now that $M$ has a projective class.  Then by Theorem \ref{supersquare} there is a projective Laplacian associated with $\{\,\cdot\,,\,\cdot\,\}$  which satisfies (\ref{BVgen}).

The algebra of functions $C^{\infty}(M)$ is then a Batalin-Vilkovisky algebra if and only if $\Delta^{2} = 0$.  This condition can be explored, as in \cite{oloii}, by considering the order of $\Delta^{2}$ as a differential operator: as $\Delta$ is constant free, $\ord\Delta^{2} = 0 \Leftrightarrow \Delta^{2} = 0$.  Since $\ord\Delta\le 2$, clearly $\ord\Delta^{2}\le 4$.  By virtue of the fact $\Delta$ is odd, in fact, $\ord\Delta^{2}\le 3$.  The bracket $[\,\cdot\,,\,\cdot\,]$ satisfying the Jacobi identity is equivalent to $\ord \Delta^{2}\le 2$.  There are two final conditions for the operator to square to zero.  For the projective Laplacian (\ref{superprojectivelaplacian}), the aforementioned translate into the following theorem.

\begin{theorem}Let $M$ be an $n|m$-dimensional supermanifold with a projective class $\Pi^{k}_{ij}$ and an odd Poisson bracket of functions $[f,g] = (-1)^{\f}((S,f),g)$ defined by a master Hamiltonian $S^{ij}p_{j}p_{i}$.  Then the projective Laplacian is a Batalin-Vilkovisky operator if and only if
\begin{equation}\label{BVformulae}
\begin{split}
S^{jk}\partial_{k}\partial_{j} + T^{j}\partial_{j}T^{i} &= 0\\
S^{kl}\partial_{l}\partial_{k}S^{ij} + T^{k}\partial_{k}S^{ij} +S^{ik}\partial_{k}T^{j}(-1)^{j} + S^{jk}\partial_{k}T^{i}(-1)^{\i(\j+1)} &= 0,
\end{split}
\end{equation}
where \[T^{i} = \frac{2}{n-m+3}\partial_{j}S^{ji} - \frac{n-m+1}{n-m+3}S^{jk}\Pi_{kj}^{i}.\]
\end{theorem}
\subsection{Projective Extensions of odd Poisson brackets}
\begin{theorem}\label{oloiibrackext}(\cite{oloii})Let $M$ be a manifold with an odd bracket of weight $0$ on densities specified by a triple $(S^{ij},\gamma^{i},\theta)$ as in (\ref{bracketcomponents}).  The associated antisymmetric bracket satisfies the Jacobi identity if and only if
\begin{eqnarray}
(S,S) = 0,\label{brackfirst}\\
(S,\gamma) = 0\label{bracksecond},\\
(S,\theta) +(\gamma,\gamma) = 0,\label{brackthird}\\
(\gamma,\theta) = 0\label{brackfourth}
\end{eqnarray}
where $S = S^{ij}p_{j}p_{i}$, $\gamma = \gamma^{i}p_{i}$ are components of the master Hamiltonian.
\end{theorem}
In the case of brackets of weight zero, $S^{ij}$ is a tensor and therefore defines a bracket on $M$.  The first condition is then equivalent to the Jacobi identity for the induced antisymmetric bracket on $M$.  It also implies that the operator $D = (S,\,\cdot\,)$ is a differential acting on functions on $T^{*}M$ (the `Lichernowicz operator').  The components $\gamma^{i}$ can be interpreted as the coefficients for an upper connection on the bundle of volume forms over the tensor $S^{ij}$.  The second condition then becomes a flatness condition for this differential.

In the special case of $S^{ij}$ being a non-degenerate matrix, condition (\ref{brackfourth}) is automatically satisfied and (\ref{brackfirst})-(\ref{brackthird}) can be rephrased as follows.
\begin{theorem}(\cite{oloii})Let $M$ be a manifold as in Theorem \ref{oloiibrackext} but in addition let $S^{ij}$ be non-degenerate.  Then the Jacobi identity for the associated antisymmetric bracket is equivalent to the conditions
\begin{enumerate}
\item The antisymmetric bracket on $C^{\infty}(M)$ makes $M$ an odd symplectic manifold.
\item For some local volume form $\rho$, $\gamma^{i} = -S^{ij}\partial_{j}\log \rho$.
\item We have $\theta = \gamma^{k}\gamma_{k}$, where $\gamma_{i} = -\partial_{i}\log\rho$ is the connection in the volume bundle associated with $\rho$.
\end{enumerate}
\end{theorem}
Let us endow $M$ with a projective class.  Then Corollary \ref{superbracketextension} gives a canonical bracket on densities associated to any bracket on $M$.  The above theorem can then be used to obtain:
\begin{theorem}Let $M$ be an $n|m$-dimensional supermanifold with an odd symplectic structure whose odd Poisson bracket is specified by a master Hamiltonian $S = S^{ij}p_{j}p_{i}$ by $[f,g] = (-1)^{\f}((S,f),g)$ and with a projective class $\Pi^{k}_{ij}$.  The corresponding canonical bracket on densities defines an odd Poisson bracket if and only if there is a local volume form $\rho$ such that
\begin{equation}\label{bracketconditionformulae}
\begin{split}
S^{jk}\Pi^{i}_{kj} &= -\left(\partial_{j}S^{ji} +\frac{n-m+3}{n-m+1}S^{ij}\partial_{j}\log\rho\right)\\
S^{ij}B_{ji} &= -\partial_{j}(S^{ji}\partial_{i}\log\rho) - \frac{n-m+2}{n-m+1}S^{ij}\partial_{i}\log\rho\partial_{j}\log\rho(-1)^{j}
\end{split}
\end{equation}
where
\[B_{kj} = \frac{n-m + 1}{n-m-1}\left(\partial_{q}\Pi^{q}_{kj} + \Pi^{p}_{qk}\Pi^{q}_{pj}\right)(-1)^{\q(1+\i+\j)}.\]
\end{theorem}
\section{Further Discussion}
There are a great number of questions raised and left unanswered by this paper.

\subsubsection*{Geometric interpretations} The quantities $B_{bc}$ can be throught of as a projective version of the Ricci tensor and therefore $S^{ij}B_{ji}$ as a projective version of scalar curvature.  As yet however, there is no elegant geometric interpretation of the formulae (\ref{BVformulae}) and (\ref{bracketconditionformulae}) in terms of the projective class and the Poisson structure.  Such an interpretation may come in terms of the associated path spaces (unparametrised geodesics).
\subsubsection*{Lie algebroids and Cartan connections} The projective classes considered here are one potential definition of projective connection.  Another possible definition is that of a \emph{projective Cartan geometry}.  These consist of a Lie group $G$ with a closed subgroup $H$ along with a principal $H$-bundle $P\to M$ and a $1$-form on $P$ taking values in the Lie algebra of $G$. In \cite{blaom, crampin09}, from such a geometry, a Lie algebroid is constructed.  There is a by now well known link between homological vector fields, Lie algebroids and odd Poisson structures studied extensively in \cite{Graded,vaintrob,roytenberg99}.  In the case of projective Cartan geometries, where $G$ is taken to be the projective general linear group and $H$ an isotropy subgroup, the link between odd Poisson structures and projective classes studied here and the associated Lie algebroids certain merits consideration.

\subsection*{Acknowledgements:}  This work would not have been possible without the constant support and advice of my teachers Hovhannes Khudaverdian and Ted Voronov, for which I am most grateful.  The precise formulae for the coefficients of the projective class were calculated with my friend Oliver Little who also kindly offered himself as a sounding board for many of the other ideas contained here.
\bibliography{compare}
\bibliographystyle{plain}

\end{document}